\documentclass[11pt]{article}

\usepackage{amsmath,amssymb}
\usepackage{tikz}
\usepackage{amsthm}

\title{On the metric dimension of incidence graph of M\"obius planes}
\date{}

\author{\'Akos Beke}

\newtheorem{theorem}{Theorem}[section]
\newtheorem{lemma}[theorem]{Lemma}

\newtheorem{proposition}[theorem]{Proposition}

\newtheorem{definition}[theorem]{Definition}

\newtheorem{example}[theorem]{Example}

\begin{document}

\maketitle

\begin{abstract}
We study the metric dimension and optimal split-resolving sets of the point-circle
incidence graph of
a M\"obius plane. We prove that the metric dimension of a M\"obius plane of order $q$ is
around $2q$, and that an optimal split-resolving set has cardinality between approximately
$5q$ and $2.5q\log q$. We also prove that a smallest blocking set of a M\"obius plane of
order $q$ has at most $2q(1 + \log(q + 1))$ points.
\end{abstract}

\section{Introduction}

The concept of metric dimension can be discussed in any metric
space, and it already appeared in 1953  \cite{AppOfDisGeo}.
In graph theory, resolving sets and metric dimension were first introduced
independently by Slater \cite{ResSetSlater}, and Harary and Melter
\cite{ResSetHararyMelter}.
The topic has been studied in several articles, and many results have been gathered
in \cite{BasMetDim} and \cite{ResSetMetDim}.
Since then, the metric dimension of various graph classes have been studied, including numerous
graphs arising from finite geometries
\cite{Bayley,MetDimAffBiaffGenQuad,ResSetHigDimProj,ResSetHigDimProjPL,HegTakResSet}.

In this paper we study the point-circle incidence graphs of M\"obius planes and give
lower and upper bounds for
the metric dimension and for the size of smallest split-resolving sets of the incidence graphs of
M\"obius planes.

\begin{definition} Let $G=(V, E)$ be a graph. We say that a set $W\subseteq V$ is \emph{resolved
by the set $S\subseteq V$} if for any two different vertices $v,u\in W$ there
is a vertex $s\in S$ such that $d(v,s)\ne d(u,s)$.

$S$ is called a \emph{resolving set} of $G$ if it resolves the set $V$. The cardinality of
a smallest resolving set is called the \emph{metric dimension of the graph} and it is
denoted by $\mu(G)$.

Let $G=(V,E)$ be a bipartite graph with vertex classes $A$ and $B$. We say that $S$ is a
\emph{split-resolving set} if $S\cap A$ resolves $B$ and $S\cap B$ resolves
$A$.
\end{definition}

Note that if $G$ is a bipartite graph with vertex classes $A$ and $B$ and the set $S$ resolves the
classes, then $S$ is also a resolving set, because if $a\in A$
and $b\in B$, then for any element $s$ of $S$, $d(a,s)$ is odd if and only if $d(b,s)$ is
even.

In some cases we will reformulate the problem into a blocking problem of a hypergraph.
A blocking set of a hypergraph is a subset of vertices such that every edge has at least
one common vertex with the subset. The goal is to determine the size of a smallest blocking set.
The fact that a hyperedge $e\in E$ is blocked can be considered as the constraint
\[\sum_{v\in e}x_v\ge1\]
holds, where the variables $x_v$ correspond to the vertices $v$ of the hypergraph.
The objective function of this LP problem is
\[\sum_{v\in V}x_v\to \min.\]
If there are constraints for all variables $x_v$ such that $x_v\in\{0,1\}$, then a
solution of the LP corresponds to a solution for the blocking problem. If we change these
constraints into $x_v\ge0$, then the solution of the LP is called the fractional solution
of the blocking problem.

We will consider points and circles of a M\"obius plane as vertices and hyperedges of a
hypergraph. The blocking set of this hypergraph is called a blocking set of the M\"obius
plane. These kind of blocking sets have been studied in several articles (see
\cite{LowBoundBlockSet,RanConDenRes,GrefRos,BlockInvPlan,BlockSetFinPlanSpac}).
We will give an upper bound to an optimal blocking set in Theorem \ref{BlockingUpperBound}.

We will use the following theorem to give upper bounds for the considered combinatorial
problems.
\begin{theorem}[Lov\'asz \cite{Lovasz}]\label{LovThm} Let $\tau$ denote the optimum of the blocking set
problem of a hypergraph $H=(V,E)$. If $B\subseteq V$ is an optimal
blocking set of the hypergraph ($\tau=|B|$), then
\[\tau<\tau^*(1+\log(d)),\]
where $d$ is the greatest degree of the hypergraph, that is
\[d=\max\{|\{ e\in E\ :\ v\in e \}|\ :\ v\in V\}\]
and $\tau^*$ is the fractional optimal
solution.
\end{theorem}

We will use this theorem by constructing a hypergraph such that a subset of its vertices
is a blocking set of the hypergraph if and only if it resolves a particular subset of the
graph.

\section{M\"obius planes and their incidence graphs}\label{MobiusSection}

\begin{definition}\label{mobiusdef}
Let $\mathcal M=(\mathcal P, \mathcal Z)$ a hypergraph. We call this hypergraph
a \emph{M\"obius plane}, the elements of $\mathcal P$ points and the elements of
$\mathcal Z$ circles if the following axioms
hold:
\begin{enumerate}
  \item For every three pairwise different points there is exactly one circle through them.
  \item If $z\in\mathcal Z$, $P\in z$ and $Q\in\mathcal P\diagdown z$, there is exacly
        one circle $z'$ through $P$ and $Q$ such that $z\cap z'=\{P\}$.
  \item There is at least one circle, and every circle has at least three points.
  \item For every circle $z$ there is at least one point $P$ such that $P\not\in z$.
\end{enumerate}
If $|\mathcal P|<\infty$ then $\mathcal M$ is a \emph{finite M\"obius plane}.
\end{definition}

In a finite M\"obius plane, every circle has the same number of points.
If a circle has $q+1$ points, then $q$ is called
the \emph{order} of the M\"obius plane. In this case
there are $q^2+1$ points and
$q(q^2+1)$ circles in the plane, and
there are $q(q+1)$ circles through every point.
For a point $P\in\mathcal P$ let us define the sets
\[\mathcal P'=\mathcal P\diagdown\{P\},\hspace{1cm} \mathcal L=\{z\diagdown\{P\}\ :\ P\in z\in \mathcal Z\}.\]
Then the hypergraph $(\mathcal P', \mathcal L)$ is an affine plane, called the \emph{affine
residue at point $P$}.
More details and constructions of M\"obius planes can be found in \cite{DemBook}
and \cite{KissSzonyiBook}.

We give a simple example for the smallest M\"obius plane:

\begin{example}
By axioms 3 and 4, there is at least one circle $z$, three points on $z$ and a fourth
point not on $z$. By axioms 2 and 3 there are at least two points not on $z$. So there are
at least 5 points, and if there are only three points on $z$ then, by axiom 1, there are
${5 \choose 3}$ circles.

Let $\mathcal P=\{1,2,3,4,5\}$ and $\mathcal Z=\{z\subseteq \mathcal P\ :\ |z|=3\}$.

It is easy to see that $(\mathcal P, \mathcal Z)$ is a M\"obius plane of order 2, it has
five points and ten circles.
\end{example}

From now on let
$\mathcal{M}(q)=(\mathcal P, \mathcal Z)$ be a M\"obius plane of order $q$.
\\
We introduce some notation:
\begin{itemize}
  \item The set of circles which go through a point $P$ is denoted by $[P]$.
  \item For any three points $A$, $B$ and $C$ we denote the circle through them by $ABC$.
  \item We say the circle $a$ is \emph{skew to the circle $b$} if they have no common points.
  \item We say the circle $a$ is \emph{tangent to the circle $b$} if they have one common point.
\end{itemize}

In the next lemma we summarize combinatorial statements that are important for us.
\begin{lemma}\label{LemmaOfCordinalities}
Let $z\in\mathcal Z$ be a circle.
\begin{enumerate}
  \item There are $q+1$ circles through two distinct points.
  \item There are $\frac{(q+1)q^2}2$ circles with two common points with $z$.
  \item There are $q-1$ circles tangent to $z$ through a point of $z$.
  \item There are $q^2-1$ circles tangent to $z$.
  \item There are $\frac{q^3-3q^2+2q}2$ circles skew to $z$.
  \item There are $\frac{q^3+3q^2-2}2$ circles which have one or two common points with $z$.
\end{enumerate}
\end{lemma}

\begin{proof}\ 
\begin{enumerate}
  \item Let $P$ and $Q$ be two different points and $H:=\mathcal P\diagdown\{P,Q\}$.
        Consider the circles through $P$ and $Q$ and let $k$ denote the number of such
        circles. All of them covers $q-1$ points of $H$. By the first axiom, every point
        $X\in H$ is covered by exactly one of them. Hence
        \[k(q-1)=|H|=q^2-1.\]
        So there are $k=q+1$ circles through $P$ and $Q$.
  \item For every two points on $z$ there are $q$ circles through them different from $z$.
        Counting them we get ${{|z|}\choose2 }q=\frac{(q+1)q^2}2$ such circles.
  \item Let $P$ be a point on $z$. There are $q(q+1)$ circles through the point $P$. One
        of them is $z$. We can choose another point $Q$ on $z$ in $q$ different ways, and
        there are $q$ circles through $P$ and $Q$ different from $z$. Hence the number of 
        circles tangent to $z$ on $P$ is
        \[q(q+1)-1-qq=q-1.\]
  \item The circle $z$ has $q+1$ points and, by the previous statement, there are $q-1$ circles tangent to $z$
        on each of them. Thus there are $(q+1)(q-1)=q^2-1$ circles tangent to $z$.
  \item There are $q(q^2+1)$ circles, one of them is $z$ itself. Subtracting the number
        of circles which have at least one common point with $z$, we get
        \[q(q^2+1)-1-(q^2-1)-\frac{(q+1)q^2}2=\frac{q^3-3q^2+2q}2.\]
  \item Adding the number of circles with one or two common points with $z$ we get
        \[\frac{(q+1)q^2}2+q^2-1=\frac{q^3+3q^2-2}2.\]
\end{enumerate}

\end{proof}

\begin{definition}\label{incgraphdef}
 The \emph{point-circle incidence graph of a M\"obius plane}
$\mathcal{M}(q)$ is $G=(V,E)$, where
$V:=\mathcal P\cup \mathcal Z$ and $E:=\{\{P,z\}\ :\ P\in z \}$.
\end{definition}

This is obviously a bipartite graph with vertex classes $\mathcal P$ and $\mathcal Z$.
The metric dimension of $G$ will be considered as the metric dimension of the geometry and
we use the notation $\mu(\mathcal M(q))$ instead of $\mu(G)$.
For every $P,Q\in\mathcal P$ and $a, b\in\mathcal Z$ we have
\[d(P,Q)=\left\{
  \begin{array}{ll}
    0 & \textrm{ if } P=Q\\
    2 & \textrm{ if } P\ne Q
  \end{array}
\right.
\hspace{2cm}d(a, b)=\left\{
  \begin{array}{ll}
    0 & \textrm{ if } a=b\\
    2 & \textrm{ if } a\cap b\ne\emptyset\\
    4 & \textrm{ if } a\cap b=\emptyset
  \end{array}
\right.\]
\[d(a,P)=\left\{
  \begin{array}{ll}
    1 & \textrm{ if } P\in a\\
    3 & \textrm{ if } P\not\in a
  \end{array}
\right.\]

\begin{definition}
For a subset $S\subseteq V$ we call the circles \emph{outer circles} and the points \emph{outer
points} if they are not elements of $S$.
\end{definition}

\section{Metric dimension of M\"obius planes}

We give a construction that resolves the set of points, and then we will use Theorem
\ref{LovThm} to find an upper bound of minimal cardinality of a set that resolves the
set of circles.

Let us construct a hypergraph $H=(V,E')$ with the same vertex set as $G$. For every two
different circles $a$ and $b$ we construct a hyperedge $e_{a, b}$ which contains all
vertices $v$ for which $d(v,a)\ne d(v,b)$. That means
\begin{itemize}
  \item $a,b\in e_{a, b}$,
  \item a circle $c\in\mathcal Z\diagdown\{a,b\}$ is an element of $e_{a, b}$ if and only
        if $c$ has a common point with $a$ or $b$ but not with both,
  \item a point $P\in\mathcal P$ is in $e_{a, b}$ if and only if $P$ is incident with $a$
        or $b$ but not with both.
\end{itemize}
By definition, $e_{a,b}$ denotes the same hyperedge as $e_{b,a}$.
Any subset of the graph $G$ resolves the set of circles if and only if it is a blocking
set of the hypergraph $H$.

\begin{lemma}\label{HyperedgeLowerBound}
There are at least $\frac{q^3}2-3q^2+\frac{11q}2-1$ vertices in any hyperedge of $H$.
\end{lemma}

\begin{proof}
Let $e_{a,b}$ be a hyperedge. There are three cases depending on how many common points
$a$ and $b$ have.

\begin{figure}[h]
\begin{center}
\begin{tikzpicture}
    \filldraw[black] (0.4,2.57) circle (2pt);
    \node at (0.5,2.8) {$P$};

    \filldraw[black] (-.15,0.6) circle (2pt);
    \node at (-.25,.3) {$Q$};

    \draw (-1.3,2) circle (1.8cm);
    \node at (-3,3) {$a$};

    \draw ( 1.3,1.25) circle (1.6cm);
    \node at (3,2) {$b$};
    
    \draw (-1.7,1.2) circle (1.4cm);
    \node at (-2.8,.1) {$c$};

    \draw (0.3,2.8) circle (1cm);
    \node at (1.2,3.5) {$d$};
\end{tikzpicture}
\caption{}\label{HypGrEdgeBoundFig}
\end{center}
\end{figure}
First let $a$ and $b$ be two circles which have two common points $P$ and $Q$, and let
$Z$ denote the set of circles with two common points with $a$ containing neither $P$ nor
$Q$. Then $|Z|={{q-1}\choose2 }q$.
Among them, there are at most $(q-1)^2$ circles tangent to $b$ (like circle $c$ in Figure
\ref{HypGrEdgeBoundFig}).
Since
\[\Big|\{(\{A,B\},C)\ :\ A,B\in b,\ A\ne B,\ C\in a,\ ABC\in Z,\ A,B,C\not\in\{P,Q\}\}\Big|\le\]
\[\le\frac12{{q-1}\choose2 }(q-1),\]
we have at most
$\frac12{{q-1}\choose2 }(q-1)$ circles in $Z$ with two common points with $b$ (like circle
$d$ in Figure \ref{HypGrEdgeBoundFig}). So there are at
least ${{q-1}\choose2 }q-(q-1)^2-\frac12{{q-1}\choose2 }(q-1)$
circles with two common points with $a$ skew to $b$.
There are at least the same number of circles that have two common points with $b$ skew
to $a$. All of them are elements of the hyperedge $e_{a,b}$. The points of $a$ and $b$
except $P$ and $Q$ and the circles $a$ and $b$ are in $e_{a,b}$, too. Therefore
\[|e_{a,b}|\ge2\left({{q-1}\choose2 }q-\frac12{{q-1}\choose2 }(q-1)-(q-1)^2\right)+2(q-1)+2\]
\[=\frac{q^3}2-3q^2+\frac{11q}2-1.\]

If $a$ and $b$ have one common point, then in the same way the lower bound for the edge is
\[|e_{a,b}|\ge2\left({{q}\choose2 }q-\frac12{q \choose2 }q-q(q-1)\right)+2q+2=\frac{q^3}2-\frac{5q^2}2+4q+2.\]

Finally, if $a$ and $b$ have no common points, then the lower bound is
\[|e_{a,b}|\ge\left({{q+1}\choose2 }q-\frac12{{q+1}\choose2 }(q+1)-(q+1)(q-1)\right)+2(q+1)+2=\]
\[=\frac{q^3}2-2q^2+\frac{3q}2+6.\]

It is easy to see that the first case is the smallest of the three lower bounds.
\end{proof}

\begin{theorem}\label{MetDimUpperBound} If $q\ge 4$, then
\[\mu(\mathcal{M}(q))\le2q-2+\left(2+\frac{14q^2-20q+6}{q^3-6q^2+11q-2}\right)\left(1+\log\left(\frac{q^6}4\right)\right).\]
If $q\ge 156$, then
\[\mu(\mathcal{M}(q))\le2q+12\log(q).\]
\end{theorem}

\begin{proof}
We give a construction that resolves the set of points if $q$ is at least 3.
Let $P$ be a point and let us consider the affine residue at point $P$. Let $P_1$ and
$P_2$ be two different parallel classes in this affine plane (See Figure
\ref{ResSetForPoints}).

These are circle classes in the
M\"obius plane such that any two circles in a class are tangent to each other in the
point $P$. Let $a\in P_1$ and $b\in P_2$. We show that the set $S_1=P_1\cup P_2\diagdown\{a,b\}$
resolves the set of points.
\begin{figure}[h]
\begin{center}
\begin{tikzpicture}
    \draw (-.8,.4) to[out=160,in=340] (-.5,1.4);
    \draw (-.5,1.4) to[out=340,in=160] (.9,2);
    \node at (-1.2,1.5) {$P_1\diagdown\{a\}$};

    \draw (-.4,0) to[out=290,in=110] (1.5,-.5);
    \draw (1.5,-.5) to[out=70,in=250] (3.4,0);
    \node at (1.5,-.8) {$P_2\diagdown\{b\}$};

    \draw [gray,dashed](1,2) to[out=45,in=160] (6,3);
    \draw [gray,dashed](-1,0) to (1,2);

    \draw (2,2) to[out=45,in=170] (6,3);
    \draw (0,0) to (2,2);

    \draw (3,2) to[out=45,in=180] (6,3);
    \draw (1,0) to (3,2);

    \draw (4,2) to[out=45,in=190] (6,3);
    \draw (2,0) to (4,2);

    \draw (5,2) to[out=45,in=200] (6,3);
    \draw (3,0) to (5,2);

    \draw [gray,dashed](5,.2) to[out=0,in=290] (6,3);
    \draw [gray,dashed](-1,.2) to (5,.2);

    \draw (5,.6) to[out=0,in=280] (6,3);
    \draw (-.6,.6) to (5,.6);

    \draw (5,1) to[out=0,in=270] (6,3);
    \draw (-.2,1) to (5,1);

    \draw (5,1.4) to[out=0,in=260] (6,3);
    \draw (.2,1.4) to (5,1.4);

    \draw (5,1.8) to[out=0,in=250] (6,3);
    \draw (.6,1.8) to (5,1.8);

    \filldraw[black] (6,3) circle (2pt);
    \node at (6.2,3.2) {$P$};

\end{tikzpicture}
\caption{}\label{ResSetForPoints}
\end{center}
\end{figure}

Let $A,B\in\mathcal P\diagdown\{P\}$ be two different points. If $ABP\not\in P_1$ then $A$
and $B$ lie on two different circles of $P_1$, and if $ABP\in P_1$ then $A$ and $B$ lie
on two different circles of $P_2$.
Since $A$ is on at most two lines of $S_1$ and
the point $P$ lies on every circle of $S_1$ and $|S_1|=2q-2>2$
thus there is a circle incidence with
$P$ but not with $A$.
Therefore the set of points is resolved by the set $S_1$,
and this set has $2q-2$ elements.

If $z_1,z_2\in\mathcal Z\diagdown (P_1\cup P_2)$ and they are circles of the affine
residue at point P, then for any $s\in S_1$,
$d(z_1,s)=d(z_2,s)=2$. So the set
$S_1$ does not resolve the set $\mathcal Z$. Using
Theorem \ref{LovThm} we prove that there is a set $S_2$ with at most roughly $12\log(q)$
vertices that resolves the set of circles.

Let $a$ and $b$ be two different circles. Setting all the variables
$\frac2{q^3-6q^2+11q-2}$, by Lemma \ref{HyperedgeLowerBound} all the constraints
hold:
\[\sum_{v\in e_{a,b}}x_v\ge\left(\frac{q^3}2-3q^2+\frac{11q}2-1\right)\frac2{q^3-6q^2+11q-2}=1.\]
The object value is
\[\tau^*\le\frac2{q^3-6q^2+11q-2}(q^3+q+q^2+1)=2+\frac{14q^2-20q+6}{q^3-6q^2+11q-2}.\]
By Lemma \ref{LemmaOfCordinalities}, for every circle $x$ there are $\frac{q^3-3q^2+2q}2$
circles skew to $x$ and there are $\frac{q^3+3q^2-2}2$ circles which have one or two
common points with $x$. Also, $x$ is an element of the edge $e_{a,x}$ for every
$a\in\mathcal Z\diagdown\{x\}$. Therefore, in the hypergraph $H$, the degree of a circle
is
\[\frac{q^3-3q^2+2q}2\cdot\frac{q^3+3q^2-2}2+q^3+q-1=\]
\[\frac{q^6}4-\frac{7q^4}4+2q^3+\frac{3q^2}2-1<\frac{q^6}4.\]
It is easy to see that the degree of a point is $(q^2+q)(q^3-q^2)$. If $q\ge4$ then the degree
of a circle is greater then the degree of a point.
By Theorem \ref{LovThm}, there is a set $S_2$ with cardinality less than
\[\left(2+\frac{14q^2-20q+6}{q^3-6q^2+11q-2}\right)\left(1+\log\left(\frac{q^6}4\right)\right)\]
that resolves the set of circles, so the set $S=S_1\cup S_2$ is a resolving set.
\end{proof}

The given upper bound of the metric dimension is approximately $2q$. This approximation
holds for the lower bound too:

\begin{theorem}\label{MetDimLowerBound}
\[
\mu(\mathcal{M}(q))\ge \left\lceil 2q-4 + \frac{8}{q+2}\right\rceil \geq 2q-3. 
\]
Moreover, if $q\ge156$, then every optimal resolving set for $\mathcal{M}(q)$ contains at least $2q-4$ circles.
\end{theorem}

\begin{proof}
Let $S$ be an optimal resolving set.
Let $\mathcal Z_S$ denote the set of circles, and $\mathcal P_S$ the set of points of $S$.
Let $t$ denote the number of outer points that are covered by one circle:
\[t=|\{P\in\mathcal P\diagdown S\ :\ |[P]\cap S|=1\}|.\]
Then
$t\le|\mathcal Z_S|$, because if there were two outer points $P$ and $Q$ which are covered by
the same circle, and only by that circle, then the constraint of $\{P,Q\}$ would not be
resolved. Also, there could be only one outer point not covered by $\mathcal Z_S$. Let us
double count the set
\[\{(P,a)\in\mathcal P\times\mathcal Z\ :\ a\in\mathcal Z_S\ ,\ |[P]\cap\mathcal Z_S|\ge2\ ,\ P\in a\}\]
to obtain
\[|\mathcal Z_S|(q+1)-t\ge2(q^2+1-t-1-|\mathcal P_S|).\]
By rearranging the inequality we get
\[|\mathcal Z_S|q\ge2(q^2-t-|\mathcal P_S|)+t-|\mathcal Z_S|=2q^2-t-|\mathcal Z_S|-2|\mathcal P_S|
\ge2q^2-2|S|,\]
thus
\begin{equation}
|\mathcal Z_S|\ge2q-2\frac{|S|}q. \label{ZSlowerbound}
\end{equation}
As $|S|\geq |\mathcal Z_S|$, \eqref{ZSlowerbound} yields
\[
|S|\geq \frac{2q^2}{q+2} = 2q-4+\frac{8}{q+2},
\]
which proves the assertion on $|S|$. If $q\ge156$, we can combine \eqref{ZSlowerbound} with the upper bound in Theorem \ref{MetDimUpperBound} to obtain
\[|\mathcal Z_S|\ge2q-2\frac{2q+12\log(q)}q=2q-4-\frac{24\log(q)}q.\]
If $q\ge114$, then $24\log(q)<q$, so for $q\ge156$ we have
\[|\mathcal Z_S|\ge2q-4.\]
\end{proof}

\section{Split-resolving sets of M\"obius planes}
In this section we give a lower and an upper bound for the cardinality of an optimal
split-resolving set of a M\"obius plane.

Let $\mathcal P_S$ and $\mathcal Z_S$ denote the set of points and the set of circles of
a split-resolving set $S$.

\begin{proposition} Let $S$ be an optimal split-resolving set. If $q>5$, then
\[2q-3 \le|\mathcal Z_S|\le2q-2.\]
If $3\le q\le 5$ then
\[|\mathcal Z_S|=2q-2.\]
\end{proposition}

\begin{proof}
We can use the same construction as in the proof of Theorem \ref{MetDimUpperBound}, where
we gave a circle set which resolves $\mathcal P$ with $2q-2$ circles.

To obtain a lower bound, we can do almost the same as in the proof of Theorem
\ref{MetDimLowerBound}.
Let $t$ denote the number of points that are covered by one circle. Then
$|\mathcal Z_S|\ge t$, and there could be only one point not covered by $\mathcal Z_S$.
Let us double count the set
\[\{(P,a)\in\mathcal P\times\mathcal Z\ :\ a\in\mathcal Z_S\ ,\ |[P]\cap\mathcal Z_S|\ge2\ ,\ P\in a\}.\]
\[|\mathcal Z_S|(q+1)-t\ge2(q^2+1-t-1),\]
by using the upper bound for $t$
\[|\mathcal Z_S|(q+1)\ge2q^2-t\ge2q^2-|\mathcal Z_S|,\]
thus
\[|\mathcal Z_S|(q+2)\ge2q^2,\]
and the obtained lower bound is
\[|\mathcal Z_S|\ge2q-4+\frac8{q+2}.\]
If $q<6$ then $\frac8{q+2}>1$, therefore,
\[\textrm{if } q\in\{3,4,5\} \textrm{, then } |\mathcal Z_S|\ge2q-2,\]
\[\textrm{if } q>5 \textrm{, then } |\mathcal Z_S|\ge2q-3.\]

\end{proof}

\begin{proposition}\label{PropBoundsPS}
If $S$ is an optimal split-resolving set, then
\[3q-7 \le|\mathcal P_S|\le\frac{q+2}2\left(1+\log\left(q^5\right)\right).\]
\end{proposition}

\begin{proof}
Since the bounds are trivial for $q=2$, we can assume that $q\ge3$.
First we prove the upper bound. Let us construct a hypergraph $H=(\mathcal P,E')$ such that for
every two different circle $a$ and $b$ we construct a hyperedge $e_{a, b}$ which contains
all points $P$ for which $d(P,a)\ne d(P,b)$. That means a point $P\in\mathcal P$ is in
$e_{a, b}$ if and only if $P$ is incident with $a$ or $b$, but not with both.

Let $a$ and $b$ be two different circles. Setting all the variables to $\frac1{2q-2}$, all
the constraints hold because there are at least $2q-2$ vertices in any hyperedge of $H$.
In this case the object value is
\[\tau^*=\frac1{2q-2}(q^2+1)=\frac q2+\frac12+\frac1{q-1}\le\frac q2+1.\]

In the hypergraph $H$ the degree of a point is
$(q^2+q)(q^3-q^2)=q^5-q^3<q^5$.
By Theorem \ref{LovThm},
\[\frac{q+2}2\left(1+\log\left(q^5\right)\right)\]
vertices resolve all the constraints.

To have a lower bound, let
\[t_k=|\{z\in\mathcal Z\ :\ |z\cap S|=k\}|\hspace{1cm}k\in\{0,1,2\}\]
There can be only one unblocked circle, and for every $P\in\mathcal P_S$ there are at
most one circle that is blocked by only $P$. Thus
\[t_0\le1,\hspace{1.8cm}t_1\le|\mathcal P_S|.\]
For any two points $P_1,P_2\in\mathcal P_S$ there are at most one double bocked circle
$z$ that is blocked by $P_1$ and $P_2$. Thus
\[t_2\le{|\mathcal P_S| \choose 2}.\]
Let us double count the set
\[\{(P,z)\in\mathcal P\times\mathcal Z\ :\ P\in\mathcal P_S\ ,\ |z\cap\mathcal P_S|\ge3\ ,\ P\in z\}\]
to get
\[|\mathcal P_S|(q^2+q)-t_1-2t_2\ge 3(q^3+q-t_0-t_1-t_2).\]
By using the upper bound for $t_0$, $t_1$ and $t_2$
\[|\mathcal P_S|(q^2+q)\ge 3(q^3+q-1)-2|\mathcal P_S|-{|\mathcal P_S| \choose 2}.\]
This yields the quadratic inequality
\[|\mathcal P_S|^2+|\mathcal P_S|(2q^2+2q+3)+6(1-q^3-q)\ge 0.\]
If we substitute $3q-8$ into $|\mathcal P_S|$, we get the inequality
\[q^2+61q-46\le0.\]
Since both roots are less than 2, this inequality does not hold. Thus
\[|\mathcal P_S|\ge 3q-7.\]
\end{proof}

The corollary of the above propositions is the following.
\begin{theorem} If $S$ is an optimal split-resolving set of $\mathcal M(q)$, then
\[5q-10 \le|S|\le \frac{q+2}2\left(1+\log\left(q^5\right)\right)+2q-2.\]
\end{theorem}

Note that the bound $t_0\le1$, in the proof of Proposition \ref{PropBoundsPS} implies
that the set $\mathcal P_S$ blocks all circles with one possible exception. This
means that there is a point $P\in\mathcal P$ such that $B=\mathcal P_S\cup \{P\}$
is a blocking set of the M\"obius plane. In \cite{LowBoundBlockSet} Bruen and Rothschild
proved that if $B$ is a blocking set of the M\"obius plane of order $q\ge9$, then
$|B|\ge2q$, thus
\[\mathrm{if }\ q\ge9,\ \mathrm{ then }\ |\mathcal P_S|\ge2q-1.\]

Up to our knowledge, the best upper bound 
for the size of a blockig set in a Mobius plane of order $q$
is given by Greferath and R\"ossing in \cite{GrefRos}.
They proved that there exists a blocking set that has approximately $3q\log(q)$ points.
We prove that there exists a blocking set of size approximately $2q\log(q)$.

\begin{theorem}\label{BlockingUpperBound}
If $B$ is an optimal blocking set of $\mathcal{M}(q)$, then
\[|B|<\frac{q^2+1}{q+1}\left(1+\log(q (q+1) )\right).\]
\end{theorem}

\begin{proof}
Let $\mathcal M(q)=(\mathcal P, \mathcal Z)$ be a M\"obius plane.
We consider the points as variables and circles as constraints of an LP, that is, the
constraints are the inequalities
\[\sum_{P\in z} x_P\ge1\]
for every circle $z$.
First we give a fractional solution. Since every circles has $q+1$ points,
if we set all variables to $\frac1{q+1}$, then all constraints hold with equality. There
are $q^2+1$ variables, so the objective value is
\[\tau^*=\frac{q^2+1}{q+1}.\]
All point is incident with the same number of circles, thus the degree $d$ of the
hypergraph is the number of circles incident with a point:
\[d=q(q+1).\]
Using Theorem \ref{LovThm}, we get the upper bound
\[|B| < \tau^*(1+\log(d))=\frac{q^2+1}{q+1}\left(1+\log(q (q+1) )\right).\]
\end{proof}

\section{Results for small orders}

In this section we deal with optimal resolving sets and split-resolving sets for
M\"obius planes of small order. Let us consider first $\mathcal{M}(2)$ in detail. 
We use the construction that we gave in Section \ref{MobiusSection}.

\begin{lemma}\label{ThreePointLemma}
For any three different points $A,B$ and $C$ there is no
circle which resolves all the constraints $\{A,B\}$, $\{A,C\}$ and $\{B,C\}$.
\end{lemma}

\begin{proof}
Let us check if a circle $z$ can resolve all the considered constraints.
Without loss of generality we may assume that $A\in z$ and $B\not\in z$, so $z$ resolves
$\{A,B\}$. If $C\in z$ then $\{A,C\}$ is not resolved and if $C\not\in z$ then
$\{B,C\}$ is not resolved by $z$.
\end{proof}

We use again the notations $\mathcal P_S$ and $\mathcal Z_S$ to denote the set of points
and the set of circles of a resolving or split-resolving set $S$.

\begin{theorem} $\mu(\mathcal M(2))=4$.
\end{theorem}

\begin{proof} We prove that any four element subset of $\mathcal P$ is a
resolving set. For any two different points $P$ and $Q$ one can assume
$P\in S$, so $d(P,P)=0\ne2=d(P,Q)$. For any two different circles $a$ and $b$ there are points
$A\in a\diagdown b$ and $B\in b\diagdown a$. We can assume that $A\in S$. Then
$d(a,A)=1\ne3=d(b,A)$.

To get a lower bound for a resolving set $S$ we double count the set
\[\{(P,z)\in\mathcal P\times \mathcal Z\ :\ P\in S\ \land\ P\in z\}.\]
Every point is an element of six circles, so the set has $6|\mathcal P_S|$ elements. Note that in $\mathcal{M}(2)$, any two circles intersect, so two distinct circles cannot be resolved by a third circle. Hence there
can be at most one unblocked outer circle. So we have the inequality
\[6|\mathcal P_S| \ge 9-|\mathcal Z_S|.\]
If $|\mathcal P_S|=0$ then $|\mathcal Z_S|\ge9$ and if $|\mathcal P_S|=1$ then
$|\mathcal Z_S|\ge3$.
If $|\mathcal P_S|=2$ then by Lemma \ref{ThreePointLemma}, $|\mathcal Z_S|\ge2$. Finally,
if $|\mathcal P_S|=3$, then the constraint of the two outer points is not resolved by
$\mathcal P_S$, so $|\mathcal Z_S|\ge1$.
\end{proof}

\begin{theorem} If $S$ is an optimal split-resolving set of $\mathcal M(2)$, then
\[|\mathcal P_S|=4\hspace{.5cm}\textrm{ and }\hspace{.5cm}|\mathcal Z_S|=3.\]
\end{theorem}

\begin{proof}
We already proved that any four-element point set resolves the set of circles.
Suppose that there are at most three points in the set $\mathcal P_S$. Without loss of
generality we may assume that $1$ and $2$ are outer points. Then the circles $\{1,3,4\}$
and $\{2,3,4\}$ are at the same distance from any points of $\mathcal P_S$.

Now let us consider the set $\mathcal Z_S$.
We may assume that $\{1,2,3\}\in\mathcal Z_S$. This circle does not resolve the constraints
$\{1,2\},\{1,3\}$ and $\{2,3\}$. So by Lemma \ref{ThreePointLemma}, we need at least two
more circles.

We prove that the set $\{\{1,2,3\}, \{1,2,4\}, \{1,3,4\}\}$ resolves the set of points.
The circle $\{1,2,3\}$ resolves every constraint $\{A,B\}$, where $A\in\{1,2,3\}$ and
$B\in\{4,5\}$. So $\{4,5\}$, $\{2,3\}$, $\{1,3\}$ and $\{1,2\}$ are the constraints not resolved by
$\{1,2,3\}$. The first three is resolved by $\{1,2,4\}$ and the last one by $\{1,3,4\}$.
\end{proof}

We investigated Miquelian planes and obtained results for small orders. We used Gurobi \cite{Gurobi}
to solve the problems. The optimals of resolving sets and split resolving sets are summarized
in the following table:
\begin{center}
\begin{tabular}{|c|c|c|}
\hline
    Order of  &  Resolving     & Split-  \\
   the plane  &  set           & resolving set  \\
\hline
   3          &   8     & 11 \\
\hline
   4          & 11   &   15 \\
\hline
   5          & MIN:9 \ \ \ \ MAX:13   &  21 \\
\hline
\end{tabular}
\end{center}


\begin{thebibliography}{99}

\bibitem{Bayley}
Robert F. Bailey.
\newblock On the metric dimension of incidence graphs.
\newblock \emph{Discrete Mathematics},
341:1613--1619, 2018.

\bibitem{BasMetDim}
Robert F. Bailey, Peter J. Cameron.
\newblock Base size, metric dimension and other invariants of groups and graphs.
\newblock \emph{Bulletin of the London Mathematical Society},
43(2):209--242, 2011.

\bibitem{MetDimAffBiaffGenQuad}
Daniele Bartoli, Tam\'as H\'eger, Gy\"orgy Kiss, Marcella Tak\'ats.
\newblock On the metric dimension of affine planes, biaffine planes and generalized quadrangles.
\newblock \emph{Australasian Journal of Combinatorics},
72: 226--248, 2018.

\bibitem{ResSetHigDimProj}
Daniele Bartoli, Gy\"orgy Kiss, Stefano Marcugini, Fernanda Pambianco
\newblock Resolving sets for higher dimensional projective spaces
\newblock \emph{Finite Fields and Their Applications},
67: 2020

\bibitem{ResSetHigDimProjPL}
Daniele Bartoli, Gy\"orgy Kiss, Fernanda Pambianco
\newblock On resolving sets in the point-line incidence graph of PG(n,q)
\newblock \emph{Ars Mathematica Contemporanea},
19(2):231--247. doi:https://doi.org/10.26493/1855-3974.2125.7b0

\bibitem{AppOfDisGeo}
Leonard M. Blumenthal.
\newblock Theory and Applications of Distance Geometry.
\newblock \emph{Clarendon Press, Oxford},
2019.

\bibitem{LowBoundBlockSet}
Aiden A. Bruen, Bruce L. Rothschild.
\newblock Lower bounds on blocking sets.
\newblock \emph{Pacific Journal of Mathematics},
118(2):303--311, 1985.

\bibitem{ResSetMetDim}
Gary Chartrand, Linda Eroh, Mark A. Johnson, Ortrud R. Oellermann.
\newblock Resolvability in graphs and the metric dimension of a graph.
\newblock \emph{Discrete Applied Mathematics},
105:99–-113, 2000.

\bibitem{DemBook}
Peter Dembowski.
\newblock Finite Geometries.
\newblock \emph{Springer-Verlag},
1968.

\bibitem{RanConDenRes}
Andr\'as G\'acs, Tam\'as Sz\H onyi.
\newblock Random constructions and density results.
\newblock \emph{Designs Codes and Cryptography},
47(1):267--287, 2008.

\bibitem{GrefRos}
Marcus Greferath, Cornelia R\"osing
\newblock On the Cardinality of Intersection Sets in Inversive Planes.
\newblock \emph{Journal of Combinatorial Theory Series A},
100:181--188, 2002

\bibitem{Gurobi}
Gurobi Optimization, Gurobi Optimizer 8.1, https://www.gurobi.com.

\bibitem{ResSetHararyMelter}
Frank Harary, Robert Melter.
\newblock On the metric dimension of a graph.
\newblock \emph{Ars Combinatoria 2},
191--195, 1976.

\bibitem{HegTakResSet}
Tam\'as H\'eger, Marcella Tak\'ats.
\newblock Resolving Sets and Semi-Resolving Sets in Finite Projective Planes.
\newblock \emph{The Electronic Journal of Combinatorics},
19: 2012.

\bibitem{BlockInvPlan}
Gy\"orgy Kiss, Stefano Marcugini, Fernanda Pambianco.
\newblock On blocking sets of inversive planes.
\newblock \emph{Journal of Combinatorial Designs},
13(4):268--275, 2004.

\bibitem{KissSzonyiBook}
Gy\"orgy Kiss, Tam\'as Sz\H onyi.
\newblock Finite Geometries.
\newblock \emph{Chapman and Hall/CRC},
2019.

\bibitem{Lovasz}
L\'aszl\'o Lov\'asz.
\newblock On the ratio of optimal integral and fractional covers.
\newblock \emph{Discrete Mathematics},
13:383--390, 1975.

\bibitem{ResSetSlater}
Peter J. Slater.
\newblock Leaves of Trees.
\newblock \emph{Proceeding of the 6th Southeastern Conference on Combinatorics, Graph Theory, and Computing, Congressus Numerantium},
14:549--559, 1975.

\bibitem{BlockSetFinPlanSpac}
Tam\'as Sz\H onyi.
\newblock Blocking sets in finite planes and spaces.
\newblock \emph{Ratio Math},
5:93--106, 1992.



\end{thebibliography}
\end{document}